\documentclass[12pt]{article}
\usepackage[utf8]{inputenc}
\usepackage{amsmath}
\usepackage{amssymb}
\usepackage{xcolor}
\usepackage{geometry}
\newtheorem{proposition}{Proposition}
\newtheorem{remark}{Remark}%

\newtheorem{definition}{Definition}
\newtheorem{lemma}{Lemma}
\newtheorem{proof}{Proof}
\newtheorem{theorem}{Theorem}

\begin{document}
	\title{Stochastic Optimal Control of Decoupled Reflected FBSDEs: A Variational Approach with Residual Reflection Terms}
\author{
Hanane Ben-Gherbal\thanks{Department of Mathematics, Laboratory of Mathematical Analysis, Probability and
Optimizations, University of Biskra, Po. Box 145, 07000 Biskra, Algeria.
\texttt{h.bengherbal@yahoo.com, hananebengherbal@univ-biskra.dz}}
\and
Omar Kebiri\thanks{Department of Stochastics and its Applications, Brandenburg University of Technology Cottbus--Senftenberg, Konrad-Wachsmann-Allee 1, 03046 Cottbus, Germany.
\\ Corresponding author: \texttt{h.bengherbal@yahoo.com}}
}

\makeatletter
\renewcommand{\@date}{} 
\makeatother
	\maketitle

\textbf{Abstract:} This paper studies necessary optimality conditions for stochastic
control problems governed by decoupled forward--backward stochastic differential
equations with reflection constraints on the backward component. The cost functional
is defined through the initial value of a reflected BSDE whose lower obstacle depends
on the forward state. The nonsmooth nature of the reflection term prevents the direct
application of classical variational methods.

Using a penalization approach, spike variations, and duality arguments, we derive a
penalized variational relation. A key feature of our method is to isolate the
reflection contribution into a residual term, avoiding the differentiation of the
nonsmooth penalization operator while preserving a Hamiltonian adjoint structure analogous to the classical stochastic maximum principle. Passing to the limit yields a reflected variational inequality of
Pontryagin type, where the obstacle effect is explicitly represented through a
residual contribution associated with the reflection term.

\section{Introduction}

Backward stochastic differential equations (BSDEs), introduced by Pardoux and Peng, have become a fundamental tool in stochastic analysis, mathematical finance, stochastic control, and partial differential equations. Their extension to forward--backward stochastic differential equations (FBSDEs) provides a natural probabilistic framework for stochastic optimal control problems, where the forward component describes the state dynamics while the backward component represents the associated cost functional or value process.

The development of the stochastic maximum principle (SMP) for systems governed by
forward--backward stochastic differential equations (FBSDEs) has followed two main
directions: the decoupled and the fully coupled settings. For \textbf{decoupled
FBSDEs}, where the forward state process does not interact with the backward
components, the SMP was established through the pioneering works of
\cite{Peng1993} and \cite{PengWu1999}. These works showed that optimality
conditions could be characterized through adjoint equations, with the backward
component representing the cost functional or value process. The theory was later
extended to \textbf{fully coupled FBSDEs}, where the forward and backward
components interact mutually through their coefficients. This extension required
new analytical techniques due to the bidirectional feedback between the state and
cost dynamics. In this framework, fundamental contributions were made by
\cite{HuPeng1995} and \cite{MaYong1999}, who developed monotonicity conditions
and the continuation method to establish well-posedness. These results provided
the foundation for the derivation of Hamiltonian systems incorporating the mutual
influence of the forward and backward components and paved the way for subsequent
developments involving non-convex control domains and more general constraints.

In recent years, stochastic control problems involving nonlinear expectations, G-Brownian motion, delayed systems, jumps, and mean-field interactions have also attracted considerable attention. Several existence and optimality results have been obtained in these generalized frameworks. In particular, Kebiri and his collaborators contributed to the study of stochastic optimal control under G-expectation and generalized FBSDE systems. Existence results for coupled non-degenerate and degenerate coupled FBSDE control systems were studied in \cite{BahlaliKebiriMezerdiMtiraoui2018,BahlaliKebiriMtiraoui2017}. Relaxed optimal control problems for G-SDEs and G-FBSDEs were investigated in \cite{RedjilGherbalKebiri2023,BouananiKebiriHartmannRedjil2024}, and relaxed maximum principles under G-expectation with jumps were derived in \cite{GherbalRedjilKebiri2022}. Additional developments concerning stochastic functional differential equations, delayed systems, and applications to economics and epidemiology may be found in \cite{ElgroudBoutabiaRedjilKebiri2022,KebiriElgroud2024,LalaouiBalatifKebiri2024,EddahbiKebiriSene2023,GridAgramKebiriOksendal2025}.

On the other hand, reflected backward stochastic differential equations (RBSDEs), introduced by El Karoui et al. \cite{ElKarouiKapoudjianPardouxPengQuenez1997}, constitute an important extension of BSDEs under state inequality constraints. In an RBSDE, the backward solution is constrained to remain above a prescribed obstacle through an increasing reflection process. RBSDEs arise naturally in obstacle problems for partial differential equations, American option pricing, Dynkin games, and constrained stochastic control.

The presence of reflection introduces a major analytical difficulty in stochastic optimal control problems. Indeed, the reflection term is nonsmooth and generally prevents the direct application of classical variational arguments. As a consequence, deriving a Pontryagin-type maximum principle for reflected systems requires new analytical tools capable of handling the obstacle constraint and the singular reflection measure.

Recently, in \cite{BenGherbalKebiri2026}, the authors established a Pontryagin maximum principle for stochastic control problems governed by reflected BSDEs. Using a penalization approach combined with spike variation and duality techniques, the reflected problem was approximated by a family of Lipschitz penalized BSDEs. Passing to the limit led to an adjoint formulation in which the effect of the
reflection appears through a singular contribution associated with the active
constraint set.

The purpose of the present paper is to extend this line of research to stochastic control systems governed by decoupled forward stochastic differential equations together with reflected backward cost dynamics. More precisely, we consider a control problem where the forward component evolves according to a controlled SDE, while the cost functional is represented by a reflected BSDE constrained by a lower obstacle depending on the forward state process. The control acts both on the forward dynamics and on the generator of the backward equation.

Compared with the framework studied in \cite{BenGherbalKebiri2026}, the present setting involves an additional forward stochastic dynamics coupled with the reflected backward cost through the obstacle and generator coefficients. This creates new interactions between the forward variational equations, the reflected backward variations, and the adjoint system. In particular, the reflection contribution must be carefully separated from the smooth part of the variational structure in order to preserve a tractable Hamiltonian formulation.

To overcome these difficulties, we employ a penalization procedure for the reflected
BSDE and derive first-order variational equations for the corresponding penalized
systems through spike perturbations of the control. A key feature of our approach
is to isolate the nonsmooth reflection contribution into a residual term rather than
incorporating it into the linearized coefficients. This formulation avoids any
differentiation of the nonsmooth penalization operator and preserves the classical
structure of the adjoint equations while keeping the obstacle effect explicitly in
the limiting variational relation. Passing to the limit in the penalized problems
then yields a reflected variational inequality of Pontryagin type, where the impact
of the reflection is represented through a residual contribution associated with
the penalization procedure.

The paper is organized as follows. In Section 2, we formulate the stochastic
control problem and introduce the standing assumptions. Section 3 is devoted to
the penalized approximation of the reflected system and the corresponding
convergence results. In Section 4, we derive the spike variation equations and
establish the first-order expansion of the penalized cost functional. Section 5
introduces the adjoint equations and proves the penalized variational inequality.
Finally, passing to the limit in the penalized formulation leads to the reflected
Pontryagin-type variational inequality for the original control problem.

\section{Problem Formulation and Standing Assumptions}

Let $(\Omega,\mathcal F,\mathbb F,\mathbb P)$ be a complete filtered probability space,
with $\mathbb F=(\mathcal F_t)_{0\le t\le T}$ satisfying the usual conditions, and carrying
a $d$-dimensional Brownian motion $W=(W_t)_{0\le t\le T}$.
Let $U\subset \mathbb R^m$ be a nonempty compact set.
We denote by $\mathcal U_{ad}$ the set of all $\mathbb F$-progressively measurable processes
$u=(u_t)_{0\le t\le T}$ with values in $U$ such that
\[ 
\mathbb E\int_0^T |u_t|^2\,dt<\infty.
\]

For each admissible control $u\in\mathcal U_{ad}$, we consider the decoupled forward backward reflected system
\begin{equation}\label{state-system}
\left\{
\begin{aligned}
&dX_t^u = b(t,X_t^u,u_t)\,dt+\sigma(t,X_t^u)\,dW_t,\qquad X_0^u=x,\\[1mm]
&Y_t^u = \Phi(X_T^u)+\int_t^T f(s,X_s^u,Y_s^u,Z_s^u,u_s)\,ds +K_T^u-K_t^u-\int_t^T Z_s^u\,dW_s,\\[1mm]
&Y_t^u \ge h(t,X_t^u),\qquad 0\le t\le T,\\[1mm]
&\int_0^T \big(Y_t^u-h(t,X_t^u)\big)\,dK_t^u=0.
\end{aligned}
\right.
\end{equation}
Here:
\begin{itemize}
\item $X^u$ is the controlled forward state process,
\item $(Y^u,Z^u,K^u)$ is the solution of the reflected backward stochastic differential equation
associated with the control $u$,
\item $h(t,X_t^u)$ is the lower obstacle,
\item $K^u$ is a continuous, nondecreasing, adapted process with $K_0^u=0$ enforcing the constraint
$Y_t^u\ge h(t,X_t^u)$.
\end{itemize}

The cost functional is defined by
\begin{equation}\label{cost-functional}
J(u):=Y_0^u.
\end{equation}
Our objective is to find an optimal control $u^*\in\mathcal U_{ad}$ such that
\begin{equation}\label{optimal-control-problem}
J(u^*)=\inf_{u\in\mathcal U_{ad}}J(u).
\end{equation}

\subsection{Standing Assumptions}

Throughout the paper, we impose the following assumptions.\\
\noindent\textbf{(A1) Regularity of the forward coefficients.}
The functions
\[
b:[0,T]\times\mathbb R^n\times U\to\mathbb R^n,
\qquad
\sigma:[0,T]\times\mathbb R^n\to\mathbb R^{n\times d}
\]
are measurable and continuously differentiable with respect to the state
variable \(x\). Their derivatives \(b_x\) and \(\sigma_x\) are bounded and
uniformly continuous with respect to \(x\), uniformly in the remaining
arguments.

Moreover, there exists a constant \(C>0\) such that, for all
\[
(t,x_1,x_2,u)
\in
[0,T]\times\mathbb R^n\times\mathbb R^n\times U,
\]
\[
|b(t,x_1,u)-b(t,x_2,u)|
+
|\sigma(t,x_1)-\sigma(t,x_2)|
\le
C|x_1-x_2|,
\]
and
\[
|b(t,0,u)|
\le
C(1+|u|).
\]

\medskip
\noindent\textbf{(A2) Regularity of the generator.}
The function
\[
f:
[0,T]\times\mathbb R^n\times\mathbb R
\times\mathbb R^d\times U
\to\mathbb R
\]
is measurable and continuously differentiable with respect to
\((x,y,z)\). The partial derivatives
\[
f_x,\qquad f_y,\qquad f_z
\]
are bounded and uniformly continuous with respect to
\((x,y,z)\), uniformly in the remaining arguments.

In particular, there exists a constant \(C>0\) such that
\[
|f(t,x_1,y_1,z_1,u)-f(t,x_2,y_2,z_2,u)|
\le
C\big(
|x_1-x_2|
+
|y_1-y_2|
+
|z_1-z_2|
\big),
\]
for all admissible arguments, and
\[
|f(t,x,0,0,u)|
\le
C(1+|x|+|u|).
\]

\medskip
\noindent\textbf{(A3) Terminal condition.}
The function
\[
\Phi:\mathbb R^n\to\mathbb R
\]
is continuously differentiable, and its derivative \(\Phi_x\) is bounded
and uniformly continuous.

In particular, \(\Phi\) satisfies a linear growth condition; that is,
there exists a constant \(C>0\) such that
\[
|\Phi(x)|
\le
C(1+|x|),
\qquad
\forall x\in\mathbb R^n.
\]
Moreover, for every \(u\in\mathcal U_{ad}\),
\[
\Phi(X_T^u)
\in
L^2(\Omega,\mathcal F_T,\mathbb P).
\]

\medskip
\noindent\textbf{(A4) Obstacle.}
The function
\[
h:[0,T]\times\mathbb R^n\to\mathbb R
\]
is continuous in $t$, continuously differentiable in $x$, and $h_x$ is bounded.
In addition, for every admissible control $u\in\mathcal U_{ad}$,
\[
h(T,X_T^u)\le \Phi(X_T^u),\qquad \mathbb P\text{-a.s.}
\]

\medskip

\noindent\textbf{(A5) Regularity of the lower obstacle.}
For every admissible control \(u\in\mathcal U_{ad}\), the obstacle process
\[
S_t^u:=h(t,X_t^u),\qquad 0\le t\le T,
\]
is regular in the sense that it admits the It\^o decomposition
\[
dS_t^u=\alpha_t^u\,dt+\beta_t^u\,dW_t,
\]
where \(\alpha^u\) and \(\beta^u\) are progressively measurable and satisfy
\[
\mathbb E\int_0^T|\alpha_t^u|^2\,dt
+
\mathbb E\int_0^T|\beta_t^u|^2\,dt
<\infty .
\]
Moreover, there exists a constant \(C>0\), independent of \(u\), such that
\[
\mathbb E\int_0^T|\alpha_t^u|^2\,dt
+
\mathbb E\int_0^T|\beta_t^u|^2\,dt
\le
C\left(
1+|x|^2+
\mathbb E\int_0^T|u_t|^2\,dt
\right).
\]
\noindent\textbf{(A6) Admissible controls.}
The control region $U$ is a nonempty compact subset of $\mathbb R^m$, and
\[
\mathcal U_{ad}:=
\left\{
u:[0,T]\times\Omega\to U\ \text{progressively measurable}:
\mathbb E\int_0^T |u_t|^2\,dt<\infty
\right\}.
\]

\medskip
Under assumptions \textbf{(A1)}--\textbf{(A6)}, for each $u\in\mathcal U_{ad}$, the forward SDE admits a unique strong solution $X^u$,
and the reflected BSDE in \eqref{state-system} admits a unique solution $(Y^u,Z^u,K^u)$ in the standard spaces.
Hence the cost functional \eqref{cost-functional} is well defined.

\section{Penalization and Approximation of the Reflected Problem}

In order to handle the reflection term, we approximate the reflected BSDE in \eqref{state-system}
by a family of penalized standard BSDEs.
For each integer $n\ge1$ and each control $u\in\mathcal U_{ad}$, we consider
\begin{equation}\label{penalized-system}
\left\{
\begin{aligned}
dX_t^u &= b(t,X_t^u,u_t)\,dt+\sigma(t,X_t^u)\,dW_t,\qquad X_0^u=x,\\[1mm]
Y_t^{n,u}
&=
\Phi(X_T^u)
+\int_t^T
\Big(
f(s,X_s^u,Y_s^{n,u},Z_s^{n,u},u_s)
+n\big(h(s,X_s^u)-Y_s^{n,u}\big)^+
\Big)\,ds\\
&\hspace{5cm}
-\int_t^T Z_s^{n,u}\,dW_s.
\end{aligned}
\right.
\end{equation}
We also define the associated penalization process
\begin{equation}\label{penalization-process}
K_t^{n,u}:=\int_0^t n\big(h(s,X_s^u)-Y_s^{n,u}\big)^+\,ds,\qquad 0\le t\le T.
\end{equation}

For each $n\ge1$, the corresponding penalized cost functional is given by
\begin{equation}\label{penalized-cost}
J_n(u):=Y_0^{n,u}.
\end{equation}
For later use, we introduce the penalized generator
\begin{equation}\label{penalized-generator}
f^n(t,x,y,z,u):=f(t,x,y,z,u)+n\big(h(t,x)-y\big)^+.
\end{equation}
Then \eqref{penalized-system} may be rewritten more compactly as
\begin{equation}\label{compact-penalized-system}
Y_t^{n,u}
=
\Phi(X_T^u)
+\int_t^T f^n(s,X_s^u,Y_s^{n,u},Z_s^{n,u},u_s)\,ds
-\int_t^T Z_s^{n,u}\,dW_s.
\end{equation}

\subsection*{Well-posedness of the penalized equation}

The first step in the derivation of the variational inequality is to establish the well-posedness and convergence properties of the penalized family \eqref{penalized-system}. We begin with a uniform bound on the penalization term, then show that the constraint violation vanishes asymptotically, and finally establish uniform moment estimates for $(Y^{n,u},Z^{n,u},K^{n,u})$.

After the change of sign reducing the lower-barrier penalization to the upper-barrier scheme, the convergence analysis follows by adapting the penalization argument of ~\cite{EssakyOuknineHarraj2005}. In particular, the uniform estimates, together with the vanishing of the barrier violation and the corresponding stability estimates, yield the strong convergence of $(Y^{n,u})_{n\geq1}$ in $S^2$ and of $(Z^{n,u})_{n\geq1}$ in $H^2$. The representation of \(K^{n,u}\) through the penalized BSDE, combined with the stability estimates and the convergence of \(Y^{n,u}\) and \(Z^{n,u}\), then yields the convergence of \(K^{n,u}\) in \(S^2\).

\begin{lemma}
\label{lem:penal-control}
Assume \textbf{(A1)}--\textbf{(A6)}. Then, for every
\(u\in\mathcal U_{ad}\), there exists a constant \(C_u>0\), independent of
\(n\), such that
\begin{equation}\label{essaky-type-estimate}
\mathbb E\int_0^T
n^2\big(h(t,X_t^u)-Y_t^{n,u}\big)^{+2}\,dt
\le C_u .
\end{equation}
Moreover, \(C_u\) may be chosen such that
\[
C_u
\le
C\left(
1+|x|^2+\mathbb E\int_0^T |u_t|^2\,dt
\right),
\]
where \(C>0\) does not depend on \(n\).
\end{lemma}

\begin{proof}
Set
\[
S_t^u:=h(t,X_t^u),\qquad 0\le t\le T.
\]
The penalized BSDE reads
\[
Y_t^{n,u}
=
\Phi(X_T^u)
+
\int_t^T
\left[
f(s,X_s^u,Y_s^{n,u},Z_s^{n,u},u_s)
+
n(S_s^u-Y_s^{n,u})^+
\right]ds
-
\int_t^T Z_s^{n,u}\,dW_s .
\]

We reduce this lower-barrier penalization to an upper-barrier penalization
by the change of variables
\[
\bar Y_t^{n,u}:=-Y_t^{n,u},
\qquad
\bar Z_t^{n,u}:=-Z_t^{n,u},
\qquad
\bar U_t^u:=-S_t^u,
\qquad
\bar \xi^u:=-\Phi(X_T^u).
\]
Then
\[
(S_t^u-Y_t^{n,u})^+
=
(\bar Y_t^{n,u}-\bar U_t^u)^+ .
\]
The process \((\bar Y^{n,u},\bar Z^{n,u})\) satisfies
\[
\bar Y_t^{n,u}
=
\bar \xi^u
+
\int_t^T
\left[
\bar f(s,\bar Y_s^{n,u},\bar Z_s^{n,u},u_s)
-
n(\bar Y_s^{n,u}-\bar U_s^u)^+
\right]ds
-
\int_t^T \bar Z_s^{n,u}\,dW_s ,
\]
where
\[
\bar f(s,\bar y,\bar z,u_s)
:=
-
f(s,X_s^u,-\bar y,-\bar z,u_s).
\]
Thus, after the above change of sign, the penalization term has exactly the
same upper-barrier form as in the penalization scheme used by
\cite{EssakyOuknineHarraj2005}.

By \textbf{(A5)}, the obstacle \(S^u\) admits the Itô decomposition
\[
dS_t^u=\alpha_t^u\,dt+\beta_t^u\,dW_t,
\]
with square-integrable coefficients. Hence \(\bar U^u=-S^u\) has the same
regularity. Therefore, the penalization estimate of
\cite{EssakyOuknineHarraj2005} applies to the transformed equation and yields
\[
\mathbb E\int_0^T
n^2(\bar Y_t^{n,u}-\bar U_t^u)^{+2}\,dt
\le C_u ,
\]
where \(C_u\) is independent of \(n\).

Since
\[
(\bar Y_t^{n,u}-\bar U_t^u)^+
=
(S_t^u-Y_t^{n,u})^+
=
\big(h(t,X_t^u)-Y_t^{n,u}\big)^+,
\]
we obtain
\[
\mathbb E\int_0^T
n^2\big(h(t,X_t^u)-Y_t^{n,u}\big)^{+2}\,dt
\le C_u .
\]
Finally, the growth assumptions on \(\Phi\), \(f\), and \(h\), together with
the estimate for \(X^u\) and the regularity bound in \textbf{(A5)}, imply
\[
C_u
\le
C\left(
1+|x|^2+\mathbb E\int_0^T |u_t|^2\,dt
\right).
\]
This proves the estimate.
\end{proof}

\begin{lemma}\label{lem-vanish}
Assume \textbf{(A1)}--\textbf{(A6)}. Then, for every fixed
\(u\in\mathcal U_{ad}\),
\[
\mathbb E\int_0^T
\big(h(t,X_t^u)-Y_t^{n,u}\big)^+\,dt
\longrightarrow 0,
\qquad n\to\infty .
\]
More precisely,
\[
\mathbb E\int_0^T
\big(h(t,X_t^u)-Y_t^{n,u}\big)^+\,dt
\le
\frac{\widetilde C_u}{n},
\]
for a constant \(\widetilde C_u>0\) independent of \(n\).
\end{lemma}

\begin{proof}
From Lemma \ref{lem:penal-control}, we have
\[
\mathbb E\int_0^T
n^2\big(h(t,X_t^u)-Y_t^{n,u}\big)^{+2}\,dt
\le C_u .
\]
Hence
\[
\mathbb E\int_0^T
\big(h(t,X_t^u)-Y_t^{n,u}\big)^{+2}\,dt
\le
\frac{C_u}{n^2}.
\]
By Cauchy's inequality,
\[
\mathbb E\int_0^T
\big(h(t,X_t^u)-Y_t^{n,u}\big)^+\,dt
\le
T^{1/2}
\left(
\mathbb E\int_0^T
\big(h(t,X_t^u)-Y_t^{n,u}\big)^{+2}\,dt
\right)^{1/2}.
\]
Therefore,
\[
\mathbb E\int_0^T
\big(h(t,X_t^u)-Y_t^{n,u}\big)^+\,dt
\le
\frac{T^{1/2}C_u^{1/2}}{n}.
\]
Letting \(n\to\infty\) proves the result.
\end{proof}

\begin{proposition}[Well-posedness and uniform estimates]\label{prop-uniform-estimate}
Assume \textbf{(A1)}--\textbf{(A6)}. Then, for every
\(n\ge1\) and every \(u\in\mathcal U_{ad}\), the penalized system
\eqref{penalized-system} admits a unique solution
\[
(X^u,Y^{n,u},Z^{n,u},K^{n,u})
\in
\mathcal S^2(\mathbb R^n)
\times
\mathcal S^2(\mathbb R)
\times
\mathcal H^2(\mathbb R^d)
\times
\mathcal S^2(\mathbb R).
\]
Moreover, there exists a constant \(C>0\), independent of \(n\), such that
\begin{equation}\label{uniform-estimate-penalized}
\mathbb E\left[
\sup_{0\le t\le T}|Y_t^{n,u}|^2
+
\int_0^T |Z_t^{n,u}|^2\,dt
+
\big(K_T^{n,u}\big)^2
\right]
\le
C\left(
1+|x|^2+\mathbb E\int_0^T |u_t|^2\,dt
\right).
\end{equation}
\end{proposition}

\begin{proof}
For each \(u\in\mathcal U_{ad}\), assumption \textbf{(A1)} implies that the
forward SDE admits a unique strong solution
\(X^u\in\mathcal S^2(\mathbb R^n)\). Moreover,
\[
\mathbb E\left[
\sup_{0\le t\le T}|X_t^u|^2
\right]
\le
C\left(
1+|x|^2+\mathbb E\int_0^T |u_t|^2\,dt
\right).
\]

For fixed \(n\), the penalized generator
\[
f^n(t,x,y,z,u)
=
f(t,x,y,z,u)+n\big(h(t,x)-y\big)^+
\]
is progressively measurable and Lipschitz in \((y,z)\). Hence, by the
classical well-posedness theorem for Lipschitz BSDEs, there exists a unique
pair
\[
(Y^{n,u},Z^{n,u})
\in
\mathcal S^2(\mathbb R)\times\mathcal H^2(\mathbb R^d).
\]
We then define
\[
K_t^{n,u}
=
\int_0^t
n\big(h(s,X_s^u)-Y_s^{n,u}\big)^+\,ds .
\]

It remains to justify the estimate uniformly in \(n\). Set
\[
S_t^u:=h(t,X_t^u),\qquad 0\le t\le T.
\]
By \textbf{(A5)}, \(S^u\) is a regular lower obstacle and admits the Itô
decomposition
\[
dS_t^u=\alpha_t^u\,dt+\beta_t^u\,dW_t,
\]
with
\[
\mathbb E\int_0^T|\alpha_t^u|^2\,dt
+
\mathbb E\int_0^T|\beta_t^u|^2\,dt
\le
C\left(
1+|x|^2+\mathbb E\int_0^T|u_t|^2\,dt
\right).
\]

The penalization estimate of Lemma \ref{lem:penal-control}, which is obtained
by adapting the penalization argument of \cite{EssakyOuknineHarraj2005} through the
change of sign
\[
\bar Y^{n,u}=-Y^{n,u},
\qquad
\bar U^u=-S^u,
\]
gives
\[
\mathbb E\int_0^T
n^2\big(S_t^u-Y_t^{n,u}\big)^{+2}\,dt
\le
C\left(
1+|x|^2+\mathbb E\int_0^T |u_t|^2\,dt
\right).
\]

Moreover, applying the a priori estimates of the \cite{EssakyOuknineHarraj2005}
penalization scheme to the transformed upper-barrier equation yields
\[
\mathbb E\left[
\sup_{0\le t\le T}|Y_t^{n,u}|^2
+
\int_0^T |Z_t^{n,u}|^2\,dt
+
\big(K_T^{n,u}\big)^2
\right]
\le
C\Gamma_u,
\]
where
\[
\Gamma_u
:=
\mathbb E|\Phi(X_T^u)|^2
+
\mathbb E\int_0^T |f(t,X_t^u,0,0,u_t)|^2\,dt
+
\mathbb E\sup_{0\le t\le T}|S_t^u|^2
+
\mathbb E\int_0^T|\alpha_t^u|^2\,dt
+
\mathbb E\int_0^T|\beta_t^u|^2\,dt .
\]
Using the growth assumptions on \(\Phi\), \(f\), and \(h\), together with the
standard estimate for \(X^u\) and the regularity bound in \textbf{(A5)}, we
obtain
\[
\Gamma_u
\le
C\left(
1+|x|^2+\mathbb E\int_0^T |u_t|^2\,dt
\right).
\]
Consequently,
\[
\mathbb E\left[
\sup_{0\le t\le T}|Y_t^{n,u}|^2
+
\int_0^T |Z_t^{n,u}|^2\,dt
+
\big(K_T^{n,u}\big)^2
\right]
\le
C\left(
1+|x|^2+\mathbb E\int_0^T |u_t|^2\,dt
\right),
\]
where \(C>0\) is independent of \(n\). This completes the proof.
\end{proof}

\begin{proposition}\label{prop-penalized-convergence}
Assume \textbf{(A1)}--\textbf{(A6)}.
For every fixed control \(u\in\mathcal U_{ad}\), the sequence
\((Y^{n,u},Z^{n,u},K^{n,u})_{n\ge1}\) converges, as \(n\to\infty\), to the unique reflected solution
\((Y^u,Z^u,K^u)\) of \eqref{state-system}, in the sense that
\[
\mathbb E\Big[\sup_{0\le t\le T}|Y_t^{n,u}-Y_t^u|^2\Big]\to0,
\qquad
\mathbb E\int_0^T |Z_t^{n,u}-Z_t^u|^2\,dt\to0,
\]
and
\[
\mathbb E\Big[\sup_{0\le t\le T}|K_t^{n,u}-K_t^u|^2\Big]\to0.
\]
\end{proposition}
\begin{proof}
From Proposition \ref{prop-uniform-estimate}, we have
\[
\sup_{n\ge1}\mathbb E\Big[
\sup_{0\le t\le T}|Y_t^{n,u}|^2
+\int_0^T |Z_t^{n,u}|^2 dt
+|K_T^{n,u}|^2
\Big] < \infty,
\]
which implies that $(Y^{n,u})$ is bounded in $\mathcal S^2$ and $(Z^{n,u})$ is bounded in $\mathcal H^2$.

Since the penalized generators satisfy
\[
f^n(t,x,y,z,u)\le f^{n+1}(t,x,y,z,u),
\quad \forall (t,x,y,z,u),
\]
the comparison theorem for BSDEs yields
\[
Y_t^{n,u}\le Y_t^{n+1,u}, \quad \forall t\in[0,T],
\]
with the same terminal condition.
Consequently, $(Y^{n,u})$ is nondecreasing and uniformly bounded in $\mathcal S^2$, hence there
exists a process $Y^u$ such that
\[
Y_t^{n,u}\uparrow Y_t^u \quad \text{a.s. for all } t\in[0,T],
\]
using the standard monotone stability argument for penalized BSDEs, we obtain
\[
\mathbb E\Big[\sup_{0\le t\le T}|Y_t^{n,u}-Y_t^u|^2\Big]\longrightarrow 0.
\]

Let $\Delta Y = Y^{n,u}-Y^{m,u}$ and $\Delta Z = Z^{n,u}-Z^{m,u}$, for $n>m$. Applying Itô's
formula to $|\Delta Y_t|^2$ gives
\begin{align*}
|\Delta Y_t|^2 + \int_t^T |\Delta Z_s|^2 ds
&= 2\int_t^T \Delta Y_s \big(f(s,X_s^u,Y_s^{n,u},Z_s^{n,u},u_s)-f(s,X_s^u,Y_s^{m,u},Z_s^{m,u},u_s)\big)\,ds \\
&\quad + 2\int_t^T \Delta Y_s\, \Big[n(h(s,X_s^u)-Y_s^{n,u})^+ - m(h(s,X_s^u)-Y_s^{m,u})^+\Big] ds \\
&\quad - 2\int_t^T \Delta Y_s \Delta Z_s\, dW_s.
\end{align*}
Taking expectations removes the stochastic integral term. Using Lipschitz continuity of $f$,
\[
2\Delta Y_s\Big(
f(s,X_s^{u},Y_s^{n,u},Z_s^{n,u},u_s)
-
f(s,X_s^{u},Y_s^{m,u},Z_s^{m,u},u_s)
\Big)
\le
C|\Delta Y_s|^2+\frac{1}{2}|\Delta Z_s|^2 .
\]

For the penalization term, using the monotonicity $Y_t^{n,u}\le Y_t^{n+1,u}$ established above
together with the Lipschitz property of $x\mapsto x^+$, one shows (as in the proof of
Lemma \ref{lem-vanish}) that
\[
2\,\mathbb E\int_t^T \Delta Y_s\, \Big[n(h-Y^{n,u})^+ - m(h-Y^{m,u})^+\Big] ds
\le C\,\mathbb E\int_t^T |\Delta Y_s|^2 ds + \varepsilon_{n,m},
\]
where $\varepsilon_{n,m}\to0$ as $n,m\to\infty$, since
$\mathbb E\int_0^T (h(t,X_t^u)-Y_t^{n,u})^{+2}dt\to0$ by Lemma \ref{lem-vanish} together with
Lemma \ref{lem:penal-control}. Combining the above estimates and applying Gronwall's lemma yields
\[
\mathbb E\int_0^T |\Delta Z_s|^2 ds
\le C\, \mathbb E\int_0^T |\Delta Y_s|^2 ds + \varepsilon_{n,m}.
\]
Since $(Y^{n,u})$ is Cauchy in $\mathcal S^2$ (being monotone and bounded), it follows that
$(Z^{n,u})$ is Cauchy in $\mathcal H^2$, and there exists $Z^u\in\mathcal H^2$ such that
\[
Z^{n,u}\to Z^u \quad \text{in } \mathcal H^2.
\]

From the BSDE representation of \(K^{n,u}\),
\[
K_t^{n,u}
=
Y_t^{n,u}-Y_0^{n,u}
-\int_0^t f(s,X_s^u,Y_s^{n,u},Z_s^{n,u},u_s)\,ds
+\int_0^t Z_s^{n,u}\,dW_s,
\]
we now prove the convergence of \(K^{n,u}\) in \(\mathcal S^2\).

For \(n,m\ge1\), we have
\[
\begin{aligned}
K_t^{n,u}-K_t^{m,u}
={}&
\big(Y_t^{n,u}-Y_t^{m,u}\big)
-
\big(Y_0^{n,u}-Y_0^{m,u}\big)
\\
&-
\int_0^t
\Big[
f(s,X_s^u,Y_s^{n,u},Z_s^{n,u},u_s)
-
f(s,X_s^u,Y_s^{m,u},Z_s^{m,u},u_s)
\Big]\,ds
\\
&+
\int_0^t
\big(Z_s^{n,u}-Z_s^{m,u}\big)\,dW_s .
\end{aligned}
\]
Using the Lipschitz continuity of \(f\), the Burkholder--Davis--Gundy
inequality, and Young's inequality, we obtain
\[
\begin{aligned}
\mathbb E\left[
\sup_{0\le t\le T}|K_t^{n,u}-K_t^{m,u}|^2
\right]
\le C\,
\mathbb E\left[
\sup_{0\le t\le T}|Y_t^{n,u}-Y_t^{m,u}|^2
\right]
+
C\,\mathbb E\int_0^T
|Z_t^{n,u}-Z_t^{m,u}|^2\,dt .
\end{aligned}
\]
Since \(Y^{n,u}\) is Cauchy in \(\mathcal S^2(\mathbb R)\) and
\(Z^{n,u}\) is Cauchy in \(\mathcal H^2(\mathbb R^d)\), it follows that
\((K^{n,u})_{n\ge1}\) is Cauchy in \(\mathcal S^2(\mathbb R)\). Hence there
exists a process \(K^u\in\mathcal S^2(\mathbb R)\) such that
\[
\mathbb E\left[
\sup_{0\le t\le T}|K_t^{n,u}-K_t^u|^2
\right]\longrightarrow0.
\]
We stress that this convergence is obtained in \(\mathcal S^2\). We do not
claim any pathwise monotonicity of \(K^{n,u}\) with respect to \(n\).

Since \(K^{n,u}\to K^u\) in \(\mathcal S^2(\mathbb R)\), there exists a
subsequence, still denoted by \((K^{n,u})\), which converges uniformly on
\([0,T]\), \(\mathbb P\)-a.s., to \(K^u\). Since each \(K^{n,u}\) is continuous,
nondecreasing in \(t\), adapted, and satisfies \(K_0^{n,u}=0\), the limit
\(K^u\) admits a continuous, nondecreasing, adapted modification with
\(K_0^u=0\).

Passing to the limit in the penalized equation yields
\[
Y_t^u
=
\Phi(X_T^u)
+\int_t^T f(s,X_s^u,Y_s^u,Z_s^u,u_s)\,ds
+K_T^u-K_t^u
-\int_t^T Z_s^u\, dW_s.
\]
The passage to the limit in the minimality condition is justified by the
standard penalization argument, as in the proof of convergence for reflected
BSDEs with regular barriers. Hence,
\[
Y_t^u \ge h(t,X_t^u), \qquad
\int_0^T (Y_t^u-h(t,X_t^u))\,dK_t^u=0.
\]

Uniqueness of $(Y^u,Z^u,K^u)$ follows from the standard uniqueness theory for one-barrier
reflected BSDEs under the Lipschitz assumptions \textbf{(A1)}--\textbf{(A6)}
(El Karoui et al., 1997), applied to two solutions obtained as limits along the penalization
scheme.
\end{proof}

\section{Exact Quotient Equations and Spike Variations}
In this section, \(u^*\in\mathcal U_{ad}\) denotes an optimal control of the
original reflected control problem. For each \(n\ge1\), we consider the
penalized state system associated with the same control \(u^*\). We denote by
\[
(X^{n,*},Y^{n,*},Z^{n,*},K^{n,*})
\]
the corresponding state trajectory, that is,
\begin{equation}\label{optimal-penalized-state}
\left\{
\begin{aligned}
dX_t^{n,*}
&=
b(t,X_t^{n,*},u_t^*)\,dt
+
\sigma(t,X_t^{n,*})\,dW_t,
\qquad X_0^{n,*}=x,
\\[1mm]
Y_t^{n,*}
&=
\Phi(X_T^{n,*})
+
\int_t^T
\Big(
f(s,X_s^{n,*},Y_s^{n,*},Z_s^{n,*},u_s^*)
+
n\big(h(s,X_s^{n,*})-Y_s^{n,*}\big)^+
\Big)\,ds
\\
&\hspace{5cm}
-
\int_t^T Z_s^{n,*}\,dW_s .
\end{aligned}
\right.
\end{equation}
The penalized increasing process is given by
\begin{equation}\label{optimal-penalized-K}
K_t^{n,*}
:=
\int_0^t
n\big(h(s,X_s^{n,*})-Y_s^{n,*}\big)^+\,ds .
\end{equation}

\subsection{Spike Variation of the Control}

Let \(\tau\in[0,T)\), let \(\delta>0\) be such that
\(\tau+\delta\le T\), and let \(v\in U\) be arbitrary. For
\(\varepsilon\in(0,\delta]\), we define the spike variation of the optimal
control \(u^*\) by
\begin{equation}\label{spike-control}
u_t^{\varepsilon}
=
\begin{cases}
v, & t\in[\tau,\tau+\varepsilon],\\
u_t^*, & t\notin[\tau,\tau+\varepsilon].
\end{cases}
\end{equation}
We denote by
\[
(X^{n,\varepsilon},Y^{n,\varepsilon},Z^{n,\varepsilon},K^{n,\varepsilon})
\]
the corresponding solution of the penalized system:
\begin{equation}\label{perturbed-penalized-state}
\left\{
\begin{aligned}
dX_t^{n,\varepsilon}
&=
b(t,X_t^{n,\varepsilon},u_t^{\varepsilon})\,dt
+
\sigma(t,X_t^{n,\varepsilon})\,dW_t,
\qquad X_0^{n,\varepsilon}=x,
\\[1mm]
Y_t^{n,\varepsilon}
&=
\Phi(X_T^{n,\varepsilon})
+
\int_t^T
\Big(
f(s,X_s^{n,\varepsilon},Y_s^{n,\varepsilon},Z_s^{n,\varepsilon},
u_s^\varepsilon)
+
n\big(h(s,X_s^{n,\varepsilon})-Y_s^{n,\varepsilon}\big)^+
\Big)\,ds
\\
&\hspace{5cm}
-
\int_t^T Z_s^{n,\varepsilon}\,dW_s .
\end{aligned}
\right.
\end{equation}
Moreover,
\begin{equation}\label{perturbed-penalized-K}
K_t^{n,\varepsilon}
:=
\int_0^t
n\big(h(s,X_s^{n,\varepsilon})-Y_s^{n,\varepsilon}\big)^+\,ds .
\end{equation}

Since \(u^*\) is optimal for the reflected problem, the spike control
\(u^\varepsilon\) satisfies
\[
J(u^\varepsilon)-J(u^*)\ge0 .
\]
Moreover, the convergence of the penalization scheme yields, for every fixed
\(\varepsilon>0\),
\[
J_n(u^\varepsilon)\longrightarrow J(u^\varepsilon),
\qquad
J_n(u^*)\longrightarrow J(u^*),
\]
and consequently,
\[
J_n(u^\varepsilon)-J_n(u^*)
\longrightarrow
J(u^\varepsilon)-J(u^*),
\qquad n\to\infty .
\]
We set
\[
\Theta_t^{n,*}
:=
(X_t^{n,*},Y_t^{n,*},Z_t^{n,*}),
\qquad
\Theta_t^{n,\varepsilon}
:=
(X_t^{n,\varepsilon},Y_t^{n,\varepsilon},
Z_t^{n,\varepsilon}).
\]

Define
\begin{equation}\label{delta-b-definition}
\Delta b_t^n(v)
:=
b(t,X_t^{n,*},v)
-
b(t,X_t^{n,*},u_t^*),
\end{equation}
and
\begin{equation}\label{delta-f-definition}
\Delta f_t^n(v)
:=
f(t,\Theta_t^{n,*},v)
-
f(t,\Theta_t^{n,*},u_t^*).
\end{equation}

The contribution of the penalization term is kept in the exact quotient
form
\begin{equation}\label{penalization-quotient}
\chi_t^{n,\varepsilon}(v,\tau)
:=
\frac{n}{\varepsilon}
\left[
\big(h(t,X_t^{n,\varepsilon})
      -Y_t^{n,\varepsilon}\big)^+
-
\big(h(t,X_t^{n,*})
      -Y_t^{n,*}\big)^+
\right].
\end{equation}

We now introduce the exact quotient processes
\begin{equation}\label{exact-quotients}
\widehat X_t^{n,\varepsilon}
:=
\frac{X_t^{n,\varepsilon}-X_t^{n,*}}{\varepsilon},
\qquad
\widehat Y_t^{n,\varepsilon}
:=
\frac{Y_t^{n,\varepsilon}-Y_t^{n,*}}{\varepsilon},
\qquad
\widehat Z_t^{n,\varepsilon}
:=
\frac{Z_t^{n,\varepsilon}-Z_t^{n,*}}{\varepsilon}.
\end{equation}

For the forward coefficients, define
\begin{equation}\label{averaged-bx}
\bar b_{x,t}^{n,\varepsilon}
:=
\int_0^1
b_x\!\left(
t,
X_t^{n,*}
+\lambda(X_t^{n,\varepsilon}-X_t^{n,*}),
u_t^\varepsilon
\right)d\lambda ,
\end{equation}
and
\begin{equation}\label{averaged-sigmax}
\bar\sigma_{x,t}^{n,\varepsilon}
:=
\int_0^1
\sigma_x\!\left(
t,
X_t^{n,*}
+\lambda(X_t^{n,\varepsilon}-X_t^{n,*})
\right)d\lambda .
\end{equation}

For \(\alpha=x,y,z\), define
\begin{equation}\label{averaged-f-derivatives}
\bar f_{\alpha,t}^{n,\varepsilon}
:=
\int_0^1
f_\alpha\!\left(
t,
\Theta_t^{n,*}
+\lambda(\Theta_t^{n,\varepsilon}-\Theta_t^{n,*}),
u_t^\varepsilon
\right)d\lambda .
\end{equation}

Finally, set
\begin{equation}\label{averaged-terminal-derivative}
\bar\Phi_x^{n,\varepsilon}
:=
\int_0^1
\Phi_x\!\left(
X_T^{n,*}
+\lambda(X_T^{n,\varepsilon}-X_T^{n,*})
\right)d\lambda .
\end{equation}

By the mean-value formula, the quotient processes satisfy the following
identities exactly.

\begin{lemma}[Quotient equations]
\label{lemma-exact-quotient}
Assume \textbf{(A1)}--\textbf{(A6)}. Then, for every fixed
\(n\ge1\), \(v\in U\), and admissible spike
\((\tau,\varepsilon)\), the process
\(\widehat X^{n,\varepsilon}\) satisfies
\begin{equation}\label{exact-forward-quotient}
\left\{
\begin{aligned}
d\widehat X_t^{n,\varepsilon}
={}&
\left[
\bar b_{x,t}^{n,\varepsilon}
\widehat X_t^{n,\varepsilon}
+
\frac{1}{\varepsilon}
\Delta b_t^n(v)
\mathbf 1_{[\tau,\tau+\varepsilon]}(t)
\right]dt
\\
&+
\bar\sigma_{x,t}^{n,\varepsilon}
\widehat X_t^{n,\varepsilon}\,dW_t,
\\
\widehat X_0^{n,\varepsilon}&=0,
\end{aligned}
\right.
\end{equation}
and
\((\widehat Y^{n,\varepsilon},
\widehat Z^{n,\varepsilon})\) satisfies
\begin{equation}\label{exact-backward-quotient}
\begin{aligned}
\widehat Y_t^{n,\varepsilon}
={}&
\bar\Phi_x^{n,\varepsilon}
\widehat X_T^{n,\varepsilon}
\\
&+
\int_t^T
\Big[
\bar f_{x,s}^{n,\varepsilon}
\widehat X_s^{n,\varepsilon}
+
\bar f_{y,s}^{n,\varepsilon}
\widehat Y_s^{n,\varepsilon}
+
\bar f_{z,s}^{n,\varepsilon}
\widehat Z_s^{n,\varepsilon}
\\
&\qquad\qquad
+
\frac{1}{\varepsilon}
\Delta f_s^n(v)
\mathbf 1_{[\tau,\tau+\varepsilon]}(s)
+
\chi_s^{n,\varepsilon}(v,\tau)
\Big]ds
\\
&-
\int_t^T
\widehat Z_s^{n,\varepsilon}\,dW_s .
\end{aligned}
\end{equation}
\end{lemma}

\begin{proof}
We decompose
\[
\begin{aligned}
&
b(t,X_t^{n,\varepsilon},u_t^\varepsilon)
-
b(t,X_t^{n,*},u_t^*)
\\
&=
\Big[
b(t,X_t^{n,\varepsilon},u_t^\varepsilon)
-
b(t,X_t^{n,*},u_t^\varepsilon)
\Big]
+
\Big[
b(t,X_t^{n,*},u_t^\varepsilon)
-
b(t,X_t^{n,*},u_t^*)
\Big].
\end{aligned}
\]
The first difference is treated by the mean-value formula in the state
variable, while the second one equals
\[
\Delta b_t^n(v)
\mathbf 1_{[\tau,\tau+\varepsilon]}(t).
\]
Dividing by \(\varepsilon\) gives
\eqref{exact-forward-quotient}. The diffusion term is treated in the
same way.

Similarly,
\[
\begin{aligned}
&
f(t,\Theta_t^{n,\varepsilon},u_t^\varepsilon)
-
f(t,\Theta_t^{n,*},u_t^*)
\\
&=
\Big[
f(t,\Theta_t^{n,\varepsilon},u_t^\varepsilon)
-
f(t,\Theta_t^{n,*},u_t^\varepsilon)
\Big]
+
\Big[
f(t,\Theta_t^{n,*},u_t^\varepsilon)
-
f(t,\Theta_t^{n,*},u_t^*)
\Big].
\end{aligned}
\]
Applying the mean-value formula to the first difference and using
\[
f(t,\Theta_t^{n,*},u_t^\varepsilon)
-
f(t,\Theta_t^{n,*},u_t^*)
=
\Delta f_t^n(v)
\mathbf 1_{[\tau,\tau+\varepsilon]}(t)
\]
gives the smooth part of \eqref{exact-backward-quotient}. The difference
between the two penalization terms is exactly
\(\varepsilon\chi_t^{n,\varepsilon}(v,\tau)\).
The terminal condition follows from the mean-value formula applied to
\(\Phi\). Hence both quotient equations are exact.
\end{proof}
\begin{lemma}[Integrability of the Penalization Quotient]
\label{lemma-chi-integrability}
Assume \textbf{(A1)}--\textbf{(A6)}. Fix \(n\ge1\),
\(v\in U\), and \(\tau\in[0,T)\). Then, for every sufficiently small
\(\varepsilon>0\),
\[
\chi^{n,\varepsilon}(v,\tau)\in\mathcal H^2(\mathbb R).
\]
\end{lemma}

\begin{proof}
Fix \(n\ge1\) and \(\varepsilon>0\). Since, for fixed \(n\), the
penalized generator
\[
f^n(t,x,y,z,u)
:=
f(t,x,y,z,u)
+
n\big(h(t,x)-y\big)^+
\]
is Lipschitz continuous with respect to \((x,y,z)\), the standard
stability estimates for Lipschitz BSDEs, applied to the two penalized
solutions
\[
(Y^{n,\varepsilon},Z^{n,\varepsilon})
\quad\text{and}\quad
(Y^{n,*},Z^{n,*}),
\]
yield
\[
\begin{aligned}
&\mathbb E\left[
\sup_{0\le t\le T}
|Y_t^{n,\varepsilon}-Y_t^{n,*}|^2
+
\int_0^T
|Z_t^{n,\varepsilon}-Z_t^{n,*}|^2\,dt
\right]
\\
&\qquad\le
C_n\,
\mathbb E\left[
|X_T^{n,\varepsilon}-X_T^{n,*}|^2
+
\int_0^T
|X_t^{n,\varepsilon}-X_t^{n,*}|^2\,dt
\right]
\\
&\qquad\quad
+
C_n\,
\mathbb E\left[
\left(
\int_\tau^{\tau+\varepsilon}
|\Delta f_t^n(v)|\,dt
\right)^2
\right],
\end{aligned}
\]
where \(C_n>0\) may depend on \(n\), and
\[
\Delta f_t^n(v)
:=
f(t,X_t^{n,*},Y_t^{n,*},Z_t^{n,*},v)
-
f(t,X_t^{n,*},Y_t^{n,*},Z_t^{n,*},u_t^*).
\]

On the other hand, the standard stability estimate for the forward
SDE gives
\[
\mathbb E\left[
\sup_{0\le t\le T}
|X_t^{n,\varepsilon}-X_t^{n,*}|^2
\right]
\le
C\,
\mathbb E\left[
\left(
\int_\tau^{\tau+\varepsilon}
|\Delta b_t^n(v)|\,dt
\right)^2
\right]
<\infty,
\]
where
\[
\Delta b_t^n(v)
:=
b(t,X_t^{n,*},v)
-
b(t,X_t^{n,*},u_t^*).
\]
Therefore, for every fixed \(n\) and \(\varepsilon>0\),
\[
\mathbb E\left[
\sup_{0\le t\le T}
|\widehat X_t^{n,\varepsilon}|^2
+
\sup_{0\le t\le T}
|\widehat Y_t^{n,\varepsilon}|^2
+
\int_0^T
|\widehat Z_t^{n,\varepsilon}|^2\,dt
\right]
<\infty .
\]

Now, since the mapping \(r\mapsto r^+\) is \(1\)-Lipschitz and
\(h_x\) is bounded, we have
\[
\begin{aligned}
|\chi_t^{n,\varepsilon}(v,\tau)|
&=
\frac{n}{\varepsilon}
\left|
\big(h(t,X_t^{n,\varepsilon})-Y_t^{n,\varepsilon}\big)^+
-
\big(h(t,X_t^{n,*})-Y_t^{n,*}\big)^+
\right|
\\
&\le
\frac{n}{\varepsilon}
\left(
|h(t,X_t^{n,\varepsilon})-h(t,X_t^{n,*})|
+
|Y_t^{n,\varepsilon}-Y_t^{n,*}|
\right)
\\
&\le
n\left(
\|h_x\|_\infty
|\widehat X_t^{n,\varepsilon}|
+
|\widehat Y_t^{n,\varepsilon}|
\right).
\end{aligned}
\]
Consequently,
\[
\begin{aligned}
\mathbb E\int_0^T
|\chi_t^{n,\varepsilon}(v,\tau)|^2\,dt
&\le
2n^2\|h_x\|_\infty^2
\mathbb E\int_0^T
|\widehat X_t^{n,\varepsilon}|^2\,dt
\\
&\quad+
2n^2
\mathbb E\int_0^T
|\widehat Y_t^{n,\varepsilon}|^2\,dt
<\infty .
\end{aligned}
\]
Hence
\[
\chi^{n,\varepsilon}(v,\tau)
\in\mathcal H^2(\mathbb R).
\]

The constants appearing above may depend on \(n\). No estimate
uniform in \(n\) is required at this stage.
\end{proof}
\section{Adjoint Systems and Penalized Duality}

The quotient equations involve averaged derivatives depending on
\((n,\varepsilon)\). We therefore first introduce auxiliary adjoint
processes adapted to these coefficients. The limiting
adjoint system will be recovered after the penalization and spike limits
are taken.

\subsection{Auxiliary Adjoint Equations}

\begin{definition}\label{def-auxiliary-adjoint}
For every \(n\ge1\) and \(\varepsilon>0\), define
\(P^{n,\varepsilon}\) as the solution of
\begin{equation}\label{auxiliary-adjoint-P}
\left\{
\begin{aligned}
dP_t^{n,\varepsilon}
&=
\bar f_{y,t}^{n,\varepsilon}
P_t^{n,\varepsilon}\,dt
+
\bar f_{z,t}^{n,\varepsilon}
P_t^{n,\varepsilon}\,dW_t,
\\
P_0^{n,\varepsilon}&=1.
\end{aligned}
\right.
\end{equation}

The pair
\((p^{n,\varepsilon},q^{n,\varepsilon})\) is defined by
\begin{equation}\label{auxiliary-adjoint-p}
\left\{
\begin{aligned}
-dp_t^{n,\varepsilon}
={}&
\Big[
(\bar b_{x,t}^{n,\varepsilon})^\top
p_t^{n,\varepsilon}
+
(\bar\sigma_{x,t}^{n,\varepsilon})^\top
q_t^{n,\varepsilon}
\\
&\qquad
+
(\bar f_{x,t}^{n,\varepsilon})^\top
P_t^{n,\varepsilon}
\Big]dt
-
q_t^{n,\varepsilon}\,dW_t,
\\[1mm]
p_T^{n,\varepsilon}
&=
(\bar\Phi_x^{n,\varepsilon})^\top
P_T^{n,\varepsilon}.
\end{aligned}
\right.
\end{equation}
\end{definition}

\begin{proposition}\label{prop-auxiliary-adjoint}
Assume \textbf{(A1)}--\textbf{(A6)}. Then
\eqref{auxiliary-adjoint-P}--\eqref{auxiliary-adjoint-p}
admit unique square-integrable solutions. Moreover, there exists
\(C>0\), independent of \(n\) and \(\varepsilon\), such that
\[
\mathbb E\left[
\sup_{0\le t\le T}|P_t^{n,\varepsilon}|^2
+
\sup_{0\le t\le T}|p_t^{n,\varepsilon}|^2
+
\int_0^T|q_t^{n,\varepsilon}|^2dt
\right]
\le C.
\]
\end{proposition}

\begin{proof}
All the averaged derivatives appearing in
\eqref{auxiliary-adjoint-P}--\eqref{auxiliary-adjoint-p}
are bounded by the same bounds as the original derivatives. The result
therefore follows from the standard estimates for linear stochastic
differential equations and linear backward stochastic differential
equations.
\end{proof}

\subsection{Hamiltonian Increment and Duality}

We keep the Hamiltonian
\begin{equation}\label{penalized-hamiltonian-clean}
H(t,x,y,z,u,p,q,P)
:=
\langle p,b(t,x,u)\rangle
+
\langle q,\sigma(t,x)\rangle
+
P f(t,x,y,z,u).
\end{equation}

Since the diffusion coefficient does not depend on the control, define
the auxiliary Hamiltonian increment by
\begin{equation}\label{exact-Hamiltonian-increment}
\Delta\widehat H_t^{n,\varepsilon}(v)
:=
\left\langle
p_t^{n,\varepsilon},
\Delta b_t^n(v)
\right\rangle
+
P_t^{n,\varepsilon}
\Delta f_t^n(v).
\end{equation}

Equivalently,
\[
\begin{aligned}
\Delta\widehat H_t^{n,\varepsilon}(v)
={}&
H(t,X_t^{n,*},Y_t^{n,*},Z_t^{n,*},
v,p_t^{n,\varepsilon},q_t^{n,\varepsilon},
P_t^{n,\varepsilon})
\\
&-
H(t,X_t^{n,*},Y_t^{n,*},Z_t^{n,*},
u_t^*,p_t^{n,\varepsilon},q_t^{n,\varepsilon},
P_t^{n,\varepsilon}).
\end{aligned}
\]

\begin{lemma}[Duality identity]
\label{lemma-exact-duality}
Assume \textbf{(A1)}--\textbf{(A6)}. Then, for every
\(n\ge1\), every \(v\in U\), and every admissible spike
\((\tau,\varepsilon)\),
\begin{equation}\label{exact-duality-identity}
\frac{J_n(u^\varepsilon)-J_n(u^*)}{\varepsilon}
=
\frac1{\varepsilon}
\mathbb E\int_\tau^{\tau+\varepsilon}
\Delta\widehat H_t^{n,\varepsilon}(v)\,dt
+
\mathbb E\int_0^T
P_t^{n,\varepsilon}
\chi_t^{n,\varepsilon}(v,\tau)\,dt .
\end{equation}
\end{lemma}

\begin{proof}
Set
\[
\Gamma_t^{n,\varepsilon}
:=
P_t^{n,\varepsilon}\widehat Y_t^{n,\varepsilon}
-
\left\langle
p_t^{n,\varepsilon},
\widehat X_t^{n,\varepsilon}
\right\rangle .
\]
We apply It\^o's formula separately to the two terms defining
\(\Gamma^{n,\varepsilon}\).

First, by \eqref{exact-backward-quotient} and
\eqref{auxiliary-adjoint-P},
\[
\begin{aligned}
d\!\left(
P_t^{n,\varepsilon}\widehat Y_t^{n,\varepsilon}
\right)
={}&
P_t^{n,\varepsilon}\,d\widehat Y_t^{n,\varepsilon}
+
\widehat Y_t^{n,\varepsilon}\,dP_t^{n,\varepsilon}
+
d\left\langle
P^{n,\varepsilon},\widehat Y^{n,\varepsilon}
\right\rangle_t
\\
={}&
-P_t^{n,\varepsilon}
\bar f_x^{n,\varepsilon}(t)
\widehat X_t^{n,\varepsilon}\,dt
\\
&-
\frac{1}{\varepsilon}
P_t^{n,\varepsilon}
\Delta f_t^n(v)
\mathbf 1_{[\tau,\tau+\varepsilon]}(t)\,dt
\\
&-
P_t^{n,\varepsilon}
\chi_t^{n,\varepsilon}(v,\tau)\,dt
+dM_t^{n,\varepsilon},
\end{aligned}
\]
where the terms containing
\(\bar f_y^{n,\varepsilon}\) and
\(\bar f_z^{n,\varepsilon}\) cancel, and
\(M^{n,\varepsilon}\) is the local martingale given by
\[
M_t^{n,\varepsilon}
=
\int_0^t
P_s^{n,\varepsilon}
\left(
\widehat Z_s^{n,\varepsilon}
+
\bar f_z^{n,\varepsilon}(s)
\widehat Y_s^{n,\varepsilon}
\right)dW_s .
\]

Next, applying It\^o's formula to
\(
\langle
p_t^{n,\varepsilon},
\widehat X_t^{n,\varepsilon}
\rangle
\),
and using \eqref{exact-forward-quotient} together with
\eqref{auxiliary-adjoint-p}, we obtain
\[
\begin{aligned}
d\left\langle
p_t^{n,\varepsilon},
\widehat X_t^{n,\varepsilon}
\right\rangle
={}&
\left\langle
p_t^{n,\varepsilon},
d\widehat X_t^{n,\varepsilon}
\right\rangle
+
\left\langle
\widehat X_t^{n,\varepsilon},
dp_t^{n,\varepsilon}
\right\rangle
+
d\left\langle
p^{n,\varepsilon},
\widehat X^{n,\varepsilon}
\right\rangle_t
\\
={}&
-
P_t^{n,\varepsilon}
\bar f_x^{n,\varepsilon}(t)
\widehat X_t^{n,\varepsilon}\,dt
\\
&+
\frac{1}{\varepsilon}
\left\langle
p_t^{n,\varepsilon},
\Delta b_t^n(v)
\right\rangle
\mathbf 1_{[\tau,\tau+\varepsilon]}(t)\,dt
+dN_t^{n,\varepsilon},
\end{aligned}
\]
where the terms involving
\(\bar b_x^{n,\varepsilon}\) and
\(\bar\sigma_x^{n,\varepsilon}\) cancel, and
\(N^{n,\varepsilon}\) denotes the resulting local martingale.

Subtracting the two identities, the terms containing
\(
P_t^{n,\varepsilon}
\bar f_x^{n,\varepsilon}(t)
\widehat X_t^{n,\varepsilon}
\)
cancel. Hence
\[
\begin{aligned}
d\Gamma_t^{n,\varepsilon}
={}&
-\frac{1}{\varepsilon}
\Big[
\left\langle
p_t^{n,\varepsilon},
\Delta b_t^n(v)
\right\rangle
+
P_t^{n,\varepsilon}\Delta f_t^n(v)
\Big]
\mathbf 1_{[\tau,\tau+\varepsilon]}(t)\,dt
\\
&-
P_t^{n,\varepsilon}
\chi_t^{n,\varepsilon}(v,\tau)\,dt
+d\widetilde M_t^{n,\varepsilon},
\end{aligned}
\]
where
\[
\widetilde M_t^{n,\varepsilon}
:=
M_t^{n,\varepsilon}-N_t^{n,\varepsilon}.
\]
By the definition of
\(\Delta\widehat H_t^{n,\varepsilon}(v)\), this becomes
\[
d\Gamma_t^{n,\varepsilon}
=
-\frac{1}{\varepsilon}
\Delta\widehat H_t^{n,\varepsilon}(v)
\mathbf 1_{[\tau,\tau+\varepsilon]}(t)\,dt
-
P_t^{n,\varepsilon}
\chi_t^{n,\varepsilon}(v,\tau)\,dt
+
d\widetilde M_t^{n,\varepsilon}.
\]

At the terminal time, by
\eqref{exact-backward-quotient} and the terminal condition in
\eqref{auxiliary-adjoint-p},
\[
\widehat Y_T^{n,\varepsilon}
=
\bar\Phi_x^{n,\varepsilon}
\widehat X_T^{n,\varepsilon},
\qquad
p_T^{n,\varepsilon}
=
(\bar\Phi_x^{n,\varepsilon})^\top
P_T^{n,\varepsilon}.
\]
Therefore,
\[
\Gamma_T^{n,\varepsilon}
=
P_T^{n,\varepsilon}
\widehat Y_T^{n,\varepsilon}
-
\left\langle
p_T^{n,\varepsilon},
\widehat X_T^{n,\varepsilon}
\right\rangle
=0.
\]

Integrating the preceding differential identity over \([0,T]\) yields
\[
\begin{aligned}
\Gamma_T^{n,\varepsilon}
-
\Gamma_0^{n,\varepsilon}
={}&
-\frac{1}{\varepsilon}
\int_\tau^{\tau+\varepsilon}
\Delta\widehat H_t^{n,\varepsilon}(v)\,dt
\\
&-
\int_0^T
P_t^{n,\varepsilon}
\chi_t^{n,\varepsilon}(v,\tau)\,dt
+
\widetilde M_T^{n,\varepsilon}
-
\widetilde M_0^{n,\varepsilon}.
\end{aligned}
\]

Under the integrability estimates established above, the stochastic
integrals involved are true martingales. Consequently,
\[
\mathbb E\!\left[
\widetilde M_T^{n,\varepsilon}
-
\widetilde M_0^{n,\varepsilon}
\right]
=0.
\]
Since
\[
\widehat X_0^{n,\varepsilon}=0,
\qquad
P_0^{n,\varepsilon}=1,
\]
we have
\[
\Gamma_0^{n,\varepsilon}
=
\widehat Y_0^{n,\varepsilon}.
\]
Moreover, by the definition of the penalized cost,
\[
\widehat Y_0^{n,\varepsilon}
=
\frac{
J_n(u^\varepsilon)-J_n(u^*)
}{\varepsilon}.
\]

Taking expectations and using
\(\Gamma_T^{n,\varepsilon}=0\), we finally obtain
\[
\frac{
J_n(u^\varepsilon)-J_n(u^*)
}{\varepsilon}
=
\frac{1}{\varepsilon}
\mathbb E
\int_\tau^{\tau+\varepsilon}
\Delta\widehat H_t^{n,\varepsilon}(v)\,dt
+
\mathbb E
\int_0^T
P_t^{n,\varepsilon}
\chi_t^{n,\varepsilon}(v,\tau)\,dt,
\]
which is precisely \eqref{exact-duality-identity}.
\end{proof}

\section{Passage to the Reflected Limit}

We now pass from the exact penalized identity
\eqref{exact-duality-identity} to the reflected control problem.

For every fixed admissible control \(u\), the penalization convergence
gives
\[
(Y^{n,u},Z^{n,u},K^{n,u})
\longrightarrow
(Y^u,Z^u,K^u)
\]
in the corresponding square-integrable spaces.

In particular, for every fixed \(\varepsilon>0\),
\[
\Theta^{n,*}\longrightarrow\Theta^*,
\qquad
\Theta^{n,\varepsilon}\longrightarrow\Theta^\varepsilon,
\]
where
\[
\Theta_t^*:=(X_t^*,Y_t^*,Z_t^*),
\qquad
\Theta_t^\varepsilon
:=
(X_t^\varepsilon,Y_t^\varepsilon,Z_t^\varepsilon).
\]

\subsection{Limit of the Auxiliary Adjoint Systems}

For every fixed \(\varepsilon>0\), define
\[
\bar b_{x,t}^{\varepsilon}
:=
\int_0^1
b_x\!\left(
t,
X_t^*
+\lambda(X_t^\varepsilon-X_t^*),
u_t^\varepsilon
\right)d\lambda ,
\]
\[
\bar\sigma_{x,t}^{\varepsilon}
:=
\int_0^1
\sigma_x\!\left(
t,
X_t^*
+\lambda(X_t^\varepsilon-X_t^*)
\right)d\lambda ,
\]
and, for \(\alpha=x,y,z\),
\[
\bar f_{\alpha,t}^{\varepsilon}
:=
\int_0^1
f_\alpha\!\left(
t,
\Theta_t^*
+\lambda(\Theta_t^\varepsilon-\Theta_t^*),
u_t^\varepsilon
\right)d\lambda .
\]
Also set
\[
\bar\Phi_x^\varepsilon
:=
\int_0^1
\Phi_x\!\left(
X_T^*
+\lambda(X_T^\varepsilon-X_T^*)
\right)d\lambda .
\]

Let \(P^\varepsilon\) solve
\begin{equation}\label{epsilon-limit-adjoint-P}
\left\{
\begin{aligned}
dP_t^\varepsilon
&=
\bar f_{y,t}^{\varepsilon}
P_t^\varepsilon\,dt
+
\bar f_{z,t}^{\varepsilon}
P_t^\varepsilon\,dW_t,
\\
P_0^\varepsilon&=1,
\end{aligned}
\right.
\end{equation}
and let
\((p^\varepsilon,q^\varepsilon)\) solve
\begin{equation}\label{epsilon-limit-adjoint-p}
\left\{
\begin{aligned}
-dp_t^\varepsilon
={}&
\Big[
(\bar b_{x,t}^{\varepsilon})^\top p_t^\varepsilon
+
(\bar\sigma_{x,t}^{\varepsilon})^\top q_t^\varepsilon
+
(\bar f_{x,t}^{\varepsilon})^\top P_t^\varepsilon
\Big]dt
-
q_t^\varepsilon\,dW_t,
\\
p_T^\varepsilon
&=
(\bar\Phi_x^\varepsilon)^\top
P_T^\varepsilon .
\end{aligned}
\right.
\end{equation}

\begin{lemma}[Penalized adjoint stability for a fixed spike]
\label{lemma-fixed-spike-adjoint-stability}
For every fixed \(\varepsilon>0\),
\[
\mathbb E\left[
\sup_{0\le t\le T}
|P_t^{n,\varepsilon}-P_t^\varepsilon|^2
\right]
+
\mathbb E\left[
\sup_{0\le t\le T}
|p_t^{n,\varepsilon}-p_t^\varepsilon|^2
\right]
+
\mathbb E\int_0^T
|q_t^{n,\varepsilon}-q_t^\varepsilon|^2dt
\longrightarrow0
\]
as \(n\to\infty\).
\end{lemma}

\begin{proof}
For fixed \(\varepsilon>0\), the penalization convergence applied to
\(u^*\) and \(u^\varepsilon\) yields
\[
\Theta^{n,*}\to\Theta^*,
\qquad
\Theta^{n,\varepsilon}\to\Theta^\varepsilon
\]
in the corresponding spaces.

The boundedness and uniform continuity of the derivatives imply that
the averaged coefficients
\[
\bar b_x^{n,\varepsilon},
\quad
\bar\sigma_x^{n,\varepsilon},
\quad
\bar f_x^{n,\varepsilon},
\quad
\bar f_y^{n,\varepsilon},
\quad
\bar f_z^{n,\varepsilon},
\quad
\bar\Phi_x^{n,\varepsilon}
\]
converge in the corresponding \(L^2\)-spaces to their
\(\varepsilon\)-limit counterparts. Standard stability estimates for
linear SDEs and BSDEs then yield the result.
\end{proof}

\subsection{The Limiting Adjoint System}

The adjoint processes associated with the optimal reflected trajectory
are defined by
\begin{equation}\label{limit-adjoint-P}
\left\{
\begin{aligned}
dP_t^*
&=
f_y(t,\Theta_t^*,u_t^*)P_t^*\,dt
+
f_z(t,\Theta_t^*,u_t^*)P_t^*\,dW_t,
\\
P_0^*&=1,
\end{aligned}
\right.
\end{equation}
and
\begin{equation}\label{limit-adjoint-p}
\left\{
\begin{aligned}
-dp_t^*
&=
\Big[
b_x(t,X_t^*,u_t^*)^\top p_t^*
+
\sigma_x(t,X_t^*)^\top q_t^*
+
f_x(t,\Theta_t^*,u_t^*)^\top P_t^*
\Big]dt
-
q_t^*\,dW_t,
\\
p_T^*
&=
\Phi_x(X_T^*)^\top P_T^* .
\end{aligned}
\right.
\end{equation}

\begin{lemma}[Spike stability of the auxiliary adjoints]
\label{lemma-spike-adjoint-stability}
For every fixed \(v\in U\) and for almost every
\(\tau\in[0,T)\),
\[
\mathbb E\left[
\sup_{0\le t\le T}
|P_t^\varepsilon-P_t^*|^2
\right]
+
\mathbb E\left[
\sup_{0\le t\le T}
|p_t^\varepsilon-p_t^*|^2
\right]
+
\mathbb E\int_0^T
|q_t^\varepsilon-q_t^*|^2dt
\longrightarrow0
\]
as \(\varepsilon\downarrow0\).
\end{lemma}

\begin{proof}
The spike variation satisfies
\[
u_t^\varepsilon=u_t^*
\qquad
\text{outside }[\tau,\tau+\varepsilon].
\]
The standard stability estimates for the forward SDE and the reflected
BSDE imply
\[
X^\varepsilon\longrightarrow X^*
\quad\text{in }\mathcal S^2,
\]
and
\[
Y^\varepsilon\longrightarrow Y^*
\quad\text{in }\mathcal S^2,
\qquad
Z^\varepsilon\longrightarrow Z^*
\quad\text{in }\mathcal H^2.
\]
Hence
\[
\Theta^\varepsilon\longrightarrow\Theta^*
\]
in the corresponding spaces.

Outside the spike interval, \(u^\varepsilon=u^*\), and therefore the
boundedness and uniform continuity of the derivatives imply
\[
\bar b_x^\varepsilon
\to b_x(\cdot,X^*,u^*),
\qquad
\bar\sigma_x^\varepsilon
\to\sigma_x(\cdot,X^*),
\]
and
\[
\bar f_\alpha^\varepsilon
\to
f_\alpha(\cdot,\Theta^*,u^*),
\qquad
\alpha=x,y,z,
\]
in the corresponding \(L^2\)-spaces. The contribution of the spike
interval vanishes as its Lebesgue measure tends to zero. Similarly,
\[
\bar\Phi_x^\varepsilon
\longrightarrow
\Phi_x(X_T^*)
\quad\text{in }L^2.
\]
The conclusion follows from the standard stability estimates for the
linear adjoint equations.
\end{proof}

\subsection{Hamiltonian Limit}

Define
\begin{equation}\label{limit-Hamiltonian-increment}
\Delta H_t^*(v)
:=
H(t,X_t^*,Y_t^*,Z_t^*,v,p_t^*,q_t^*,P_t^*)
-
H(t,X_t^*,Y_t^*,Z_t^*,u_t^*,p_t^*,q_t^*,P_t^*).
\end{equation}

Since \(\sigma\) does not depend on the control,
\[
\Delta H_t^*(v)
=
\left\langle
p_t^*,
b(t,X_t^*,v)-b(t,X_t^*,u_t^*)
\right\rangle
+
P_t^*
\Big[
f(t,\Theta_t^*,v)
-
f(t,\Theta_t^*,u_t^*)
\Big].
\]
\begin{lemma}[Hamiltonian spike limit]
\label{lemma-Hamiltonian-spike-limit}
For every fixed \(v\in U\) and for almost every
\(\tau\in[0,T)\),
\begin{equation}\label{Hamiltonian-double-limit}
\lim_{\varepsilon\downarrow0}
\liminf_{n\to\infty}
\frac1{\varepsilon}
\mathbb E\int_\tau^{\tau+\varepsilon}
\Delta\widehat H_t^{n,\varepsilon}(v)\,dt
=
\mathbb E[\Delta H_\tau^*(v)].
\end{equation}
\end{lemma}

\begin{proof}
Fix \(\varepsilon>0\). By the convergence of the penalized states and
Lemma \ref{lemma-fixed-spike-adjoint-stability}, we have
\[
\begin{aligned}
&
\frac1{\varepsilon}
\mathbb E\int_\tau^{\tau+\varepsilon}
\Delta\widehat H_t^{n,\varepsilon}(v)\,dt
\\
&\qquad\longrightarrow
\frac1{\varepsilon}
\mathbb E\int_\tau^{\tau+\varepsilon}
\Delta\widehat H_t^\varepsilon(v)\,dt ,
\qquad n\to\infty,
\end{aligned}
\]
where
\[
\begin{aligned}
\Delta\widehat H_t^\varepsilon(v)
:={}&
\left\langle
p_t^\varepsilon,
b(t,X_t^*,v)-b(t,X_t^*,u_t^*)
\right\rangle
\\
&+
P_t^\varepsilon
\Big[
f(t,\Theta_t^*,v)
-
f(t,\Theta_t^*,u_t^*)
\Big].
\end{aligned}
\]

It remains to identify the limit as
\(\varepsilon\downarrow0\). For convenience, set
\[
\Delta b_t^*(v)
:=
b(t,X_t^*,v)-b(t,X_t^*,u_t^*)
\]
and
\[
\Delta f_t^*(v)
:=
f(t,\Theta_t^*,v)-f(t,\Theta_t^*,u_t^*).
\]
Then
\[
\Delta\widehat H_t^\varepsilon(v)
=
\left\langle p_t^\varepsilon,\Delta b_t^*(v)\right\rangle
+
P_t^\varepsilon\Delta f_t^*(v),
\]
whereas
\[
\Delta H_t^*(v)
=
\left\langle p_t^*,\Delta b_t^*(v)\right\rangle
+
P_t^*\Delta f_t^*(v).
\]

We decompose
\[
\begin{aligned}
&
\frac1{\varepsilon}
\mathbb E\int_\tau^{\tau+\varepsilon}
\Delta\widehat H_t^\varepsilon(v)\,dt
-
\mathbb E[\Delta H_\tau^*(v)]
\\
={}&
\frac1{\varepsilon}
\mathbb E\int_\tau^{\tau+\varepsilon}
\Big(
\Delta\widehat H_t^\varepsilon(v)
-
\Delta H_t^*(v)
\Big)\,dt
\\
&+
\left[
\frac1{\varepsilon}
\mathbb E\int_\tau^{\tau+\varepsilon}
\Delta H_t^*(v)\,dt
-
\mathbb E[\Delta H_\tau^*(v)]
\right]
\\
=:{}& I_\varepsilon+J_\varepsilon .
\end{aligned}
\]

We first consider \(I_\varepsilon\). By the preceding definitions,
\[
\begin{aligned}
\Delta\widehat H_t^\varepsilon(v)
-
\Delta H_t^*(v)
={}&
\left\langle
p_t^\varepsilon-p_t^*,
\Delta b_t^*(v)
\right\rangle
\\
&+
\big(P_t^\varepsilon-P_t^*\big)
\Delta f_t^*(v).
\end{aligned}
\]
Hence, using the supremum norm and the Cauchy--Schwarz
inequality,
\[
\begin{aligned}
|I_\varepsilon|
\le{}&
\left(
\mathbb E\sup_{0\le t\le T}
|p_t^\varepsilon-p_t^*|^2
\right)^{1/2}
\left(
\frac1{\varepsilon}
\mathbb E\int_\tau^{\tau+\varepsilon}
|\Delta b_t^*(v)|^2\,dt
\right)^{1/2}
\\
&+
\left(
\mathbb E\sup_{0\le t\le T}
|P_t^\varepsilon-P_t^*|^2
\right)^{1/2}
\left(
\frac1{\varepsilon}
\mathbb E\int_\tau^{\tau+\varepsilon}
|\Delta f_t^*(v)|^2\,dt
\right)^{1/2}.
\end{aligned}
\]
By Lemma \ref{lemma-spike-adjoint-stability},
\[
\mathbb E\sup_{0\le t\le T}
|p_t^\varepsilon-p_t^*|^2
\longrightarrow0,
\qquad
\mathbb E\sup_{0\le t\le T}
|P_t^\varepsilon-P_t^*|^2
\longrightarrow0.
\]
Moreover, for almost every \(\tau\in[0,T)\), which is a Lebesgue
point of the functions
\[
t\longmapsto
\mathbb E|\Delta b_t^*(v)|^2,
\qquad
t\longmapsto
\mathbb E|\Delta f_t^*(v)|^2,
\]
the Lebesgue differentiation theorem gives
\[
\frac1{\varepsilon}
\mathbb E\int_\tau^{\tau+\varepsilon}
|\Delta b_t^*(v)|^2\,dt
\longrightarrow
\mathbb E|\Delta b_\tau^*(v)|^2,
\]
and
\[
\frac1{\varepsilon}
\mathbb E\int_\tau^{\tau+\varepsilon}
|\Delta f_t^*(v)|^2\,dt
\longrightarrow
\mathbb E|\Delta f_\tau^*(v)|^2 .
\]
Hence these averaged terms are bounded for sufficiently small
\(\varepsilon\), and together with the adjoint stability result we obtain
\[
I_\varepsilon\longrightarrow0 .
\]

We next consider \(J_\varepsilon\). Since
\[
t\longmapsto
\mathbb E[\Delta H_t^*(v)]
\]
is integrable, the Lebesgue differentiation theorem yields, for
almost every \(\tau\in[0,T)\),
\[
\begin{aligned}
J_\varepsilon
&=
\frac1{\varepsilon}
\int_\tau^{\tau+\varepsilon}
\mathbb E[\Delta H_t^*(v)]\,dt
-
\mathbb E[\Delta H_\tau^*(v)]
\\
&\longrightarrow0,
\qquad
\varepsilon\downarrow0.
\end{aligned}
\]
Therefore,
\[
\frac1{\varepsilon}
\mathbb E\int_\tau^{\tau+\varepsilon}
\Delta\widehat H_t^\varepsilon(v)\,dt
\longrightarrow
\mathbb E[\Delta H_\tau^*(v)].
\]

Combining this limit with the fixed-\(\varepsilon\) convergence
as \(n\to\infty\), we obtain
\[
\lim_{\varepsilon\downarrow0}
\liminf_{n\to\infty}
\frac1{\varepsilon}
\mathbb E\int_\tau^{\tau+\varepsilon}
\Delta\widehat H_t^{n,\varepsilon}(v)\,dt
=
\mathbb E[\Delta H_\tau^*(v)],
\]
which proves \eqref{Hamiltonian-double-limit}.
\end{proof}

\subsection{Residual Reflection Contribution}

The contribution generated by the penalization of the reflection is
kept in residual form.

\begin{definition}\label{definition-reflection-residual}
For every fixed \(v\in U\) and for almost every
\(\tau\in[0,T)\), define
\begin{equation}\label{reflection-residual}
\mathcal R(\tau,v;u^*)
:=
\liminf_{\varepsilon\downarrow0}
\liminf_{n\to\infty}
\mathbb E\int_0^T
P_t^{n,\varepsilon}
\chi_t^{n,\varepsilon}(v,\tau)\,dt .
\end{equation}
\end{definition}

\begin{remark}\label{remark-order-limits}
The order of limits in \eqref{reflection-residual} is essential.
For every fixed \(\varepsilon>0\), the penalization scheme converges as
\(n\to\infty\) to the reflected system corresponding to the fixed spike
control \(u^\varepsilon\). Only after this penalization limit is taken
do we let the length of the spike interval tend to zero. Thus the order
\[
n\to\infty
\quad\text{first, followed by}\quad
\varepsilon\downarrow0
\]
is the one naturally induced by the optimality of the original
reflected problem.
\end{remark}

\subsection{Reflected Variational Inequality}

\begin{theorem}[Reflected variational inequality]
\label{thm-reflected-vi}
Assume \textbf{(A1)}--\textbf{(A6)}, let \(u^*\in\mathcal U_{ad}\) be an optimal control of the
reflected control problem.

Then, for every fixed \(v\in U\) and for almost every
\(\tau\in[0,T)\),
\begin{equation}\label{reflected-variational-ineq-final}
\mathbb E[\Delta H_\tau^*(v)]
+
\mathcal R(\tau,v;u^*)
\ge 0.
\end{equation}
Equivalently,
\begin{equation}\label{reflected-variational-ineq-expanded}
\begin{aligned}
&
\mathbb E\Big[
H(\tau,X_\tau^*,Y_\tau^*,Z_\tau^*,v,
p_\tau^*,q_\tau^*,P_\tau^*)
\\
&\qquad\qquad
-
H(\tau,X_\tau^*,Y_\tau^*,Z_\tau^*,u_\tau^*,
p_\tau^*,q_\tau^*,P_\tau^*)
\Big]
\\
&\qquad
+
\mathcal R(\tau,v;u^*)
\ge0 .
\end{aligned}
\end{equation}
\end{theorem}

\begin{proof}
Fix \(v\in U\) and a Lebesgue point
\(\tau\in[0,T)\) for the quantities appearing below, and let
\(u^\varepsilon\) be the corresponding spike variation.

Since \(u^*\) is optimal for the original reflected control problem,
\begin{equation}\label{reflected-optimality-spike}
J(u^\varepsilon)-J(u^*)\ge0 .
\end{equation}

For every fixed \(\varepsilon>0\), the convergence of the penalization
scheme applied separately to \(u^\varepsilon\) and \(u^*\) gives
\[
J_n(u^\varepsilon)\longrightarrow J(u^\varepsilon),
\qquad
J_n(u^*)\longrightarrow J(u^*).
\]
Consequently,
\begin{equation}\label{fixed-epsilon-cost-limit}
\liminf_{n\to\infty}
\frac{
J_n(u^\varepsilon)-J_n(u^*)
}{\varepsilon}
=
\frac{
J(u^\varepsilon)-J(u^*)
}{\varepsilon}
\ge0 .
\end{equation}

On the other hand, the exact duality identity
\eqref{exact-duality-identity} gives, for every \(n\ge1\),
\[
\begin{aligned}
\frac{
J_n(u^\varepsilon)-J_n(u^*)
}{\varepsilon}
={}&
\frac1{\varepsilon}
\mathbb E\int_\tau^{\tau+\varepsilon}
\Delta\widehat H_t^{n,\varepsilon}(v)\,dt
\\
&+
\mathbb E\int_0^T
P_t^{n,\varepsilon}
\chi_t^{n,\varepsilon}(v,\tau)\,dt .
\end{aligned}
\]

Taking \(n\to\infty\) for the fixed value of \(\varepsilon\), and using
\eqref{fixed-epsilon-cost-limit}, yields
\begin{equation}\label{fixed-epsilon-variational-limit}
\begin{aligned}
0
\le{}&
\liminf_{n\to\infty}
\Bigg[
\frac1{\varepsilon}
\mathbb E\int_\tau^{\tau+\varepsilon}
\Delta\widehat H_t^{n,\varepsilon}(v)\,dt
\\
&\qquad\qquad+
\mathbb E\int_0^T
P_t^{n,\varepsilon}
\chi_t^{n,\varepsilon}(v,\tau)\,dt
\Bigg].
\end{aligned}
\end{equation}

For each fixed \(\varepsilon>0\), Lemma
\ref{lemma-fixed-spike-adjoint-stability} and the convergence of the
penalized states imply that the Hamiltonian term has an actual limit.
Therefore,
\[
\begin{aligned}
0
\le{}&
\frac1{\varepsilon}
\mathbb E\int_\tau^{\tau+\varepsilon}
\Delta\widehat H_t^\varepsilon(v)\,dt
\\
&+
\liminf_{n\to\infty}
\mathbb E\int_0^T
P_t^{n,\varepsilon}
\chi_t^{n,\varepsilon}(v,\tau)\,dt .
\end{aligned}
\]

Taking \(\liminf_{\varepsilon\downarrow0}\) gives
\[
\begin{aligned}
0
\le{}&
\liminf_{\varepsilon\downarrow0}
\frac1{\varepsilon}
\mathbb E\int_\tau^{\tau+\varepsilon}
\Delta\widehat H_t^\varepsilon(v)\,dt
\\
&+
\liminf_{\varepsilon\downarrow0}
\liminf_{n\to\infty}
\mathbb E\int_0^T
P_t^{n,\varepsilon}
\chi_t^{n,\varepsilon}(v,\tau)\,dt .
\end{aligned}
\]

By Lemma \ref{lemma-Hamiltonian-spike-limit}, the first term converges to
\[
\mathbb E[\Delta H_\tau^*(v)],
\]
while, by Definition \ref{definition-reflection-residual}, the second
term is precisely
\[
\mathcal R(\tau,v;u^*).
\]
Hence
\[
\mathbb E[\Delta H_\tau^*(v)]
+
\mathcal R(\tau,v;u^*)
\ge0 ,
\]
which proves \eqref{reflected-variational-ineq-final}.
\end{proof}

\begin{remark}\label{remark-final-structure}
The reflected variational inequality retains the classical Hamiltonian associated with the smooth coefficients \(b,\sigma,f\), while the nonsmooth effect of the obstacle does not enter the limiting adjoint equations. Instead, it is retained through the residual term \(\mathcal R(\tau,v;u^*)\), which should be viewed as an intrinsic first-order trace of the reflection mechanism rather than as an auxiliary remainder produced by the penalization procedure.

Indeed, the nonsmooth dependence of the increasing reflection process on perturbations of the state prevents, in general, a classical linearization of the reflection operator. The penalization scheme therefore separates the smooth Hamiltonian contribution from the singular effect generated by the active obstacle, and the latter survives in the limit precisely through \(\mathcal R(\tau,v;u^*)\).

In this sense, the resulting variational inequality retains the full influence of the reflection without imposing an unjustified differentiability structure on the reflection process. A sharper characterization of \(\mathcal R(\tau,v;u^*)\), for instance in terms of a directional variation of the reflection measure or of the active contact set, would require additional regularity or structural assumptions and remains beyond the scope of the present work.
\end{remark}

\section{Conclusion}

In this paper, we have studied necessary optimality conditions for stochastic
control problems governed by decoupled forward--backward stochastic differential
equations with a reflected backward component. The main difficulty arises from
the nonsmooth nature of the reflection term, which prevents a direct application
of classical variational methods.

To overcome this difficulty, we have employed a penalization approach and derived
first-order variational equations for the penalized systems using spike
perturbations of the control. A central feature of our analysis is the separation
of the reflection contribution from the smooth linearized structure through a
residual term. This approach avoids differentiating the nonsmooth penalization
operator and allows the adjoint equations to retain the classical Hamiltonian
form of stochastic control problems.

By combining penalization convergence, variational analysis, and duality
arguments, we have obtained a reflected Pontryagin-type variational inequality
for the original control problem. The resulting formulation explicitly preserves
the influence of the obstacle through a residual reflection contribution, while
maintaining a tractable adjoint framework.

The approach developed in this work provides a useful framework for treating
reflected stochastic control problems in which the obstacle depends on the
controlled state dynamics. Future research directions include extending the
present analysis to fully coupled reflected FBSDEs, control problems with jumps
or mean-field interactions, and more general obstacle structures.

\section*{Data Availability Statement}
Data sharing not applicable to this article as no datasets were generated or analyzed during the current study.

\section*{Conflict of Interest Statement}
The authors declare that there are no conflicts of interest regarding the publication of this article.

\section*{Funding Statement}
This research did not receive any specific grant from funding agencies in the public, commercial, or not-for-profit sectors.
\bibliographystyle{plain}
\bibliography{references}

\end{document}